\documentclass[12pt]{amsart}
\usepackage{mathrsfs}
\usepackage{stmaryrd}
\usepackage{amsmath,amssymb}
\usepackage{amsfonts,txfonts}
\usepackage{amsthm}
\usepackage{latexsym}
\usepackage{graphicx}
\usepackage{color}

\def\p{\partial}

\def\Hom{\mbox{Hom}}

\def\R{\mathbb{R}}

\def\vv<#1>{\langle#1\rangle}

\def\XXint#1#2{\setbox0=\hbox{$#1{#2}{\int}$}{#2}\kern-.5\wd0 }

\def\XXint#1#2#3{{\setbox0=\hbox{$#1{#2#3}{\int}$}
     \vcenter{\hbox{$#2#3$}}\kern-.5\wd0}}

\def\vv<#1>{\left\langle#1\right\rangle}

\newtheorem{thm}{Theorem}[section]

\newtheorem{lem}{Lemma}[section]

\theoremstyle{definition}

\theoremstyle{remark}

\newtheorem{rem}{Remark}[section]

\numberwithin{equation}{section}

\def\Span{{\rm span}}

\def\Hom{{\rm Hom}}
\begin{document}

\title{Canonical connection and a geometric proof of the Frobenius theorem}

\author{Chengjie Yu$^1$}
\address{Department of Mathematics, Shantou University, Shantou, Guangdong, 515063, China}
\email{cjyu@stu.edu.cn}
\thanks{$^1$Research partially supported by GDNSF with contract no. 2025A1515011144}
\renewcommand{\subjclassname}{%
  \textup{2020} Mathematics Subject Classification}
\subjclass[2020]{Primary 53C05; Secondary 53C12}
\date{}
\keywords{integrable, distribution, canonical connection}
\begin{abstract}
In this paper, we introduce a new canonical connection on Riemannian manifold with a distribution. Moreover, as an application of the connection, we give a geometric proof of the Frobenius theorem.
\end{abstract}
\maketitle
\markboth{Chengjie Yu }{Canonical connections and a geometric proof of the Frobenius theorem}
\section{Introduction}

Let $M^n$ be a differential manifold. A smooth vector subbundle $E^r$ of $TM$ with rank $r$ is called a distribution of rank $r$. If for any $p\in M$, there is a regular submanifold $\Sigma^r$ of $M$ with dimension $r$ passing through $p$ such that $T_x\Sigma=E_x$ for any $x\in \Sigma$, then $E$ is called integrable. If for any $X,Y\in \Gamma(E)$, $[X,Y]\in \Gamma(E)$, then $E$ is called involutive. It is not hard to see that if a distribution $E$ is integrable, then it must be involutive. The Frobenius theorem for distributions confirms the converse.
\begin{thm}[Frobenius theorem]
Let $E^r$ be a distribution on $M^n$. Then, $E$ is integrable if and only if $E$ is involutive.
\end{thm}

In \cite{CW, Gu, Lu}, the authors gave alternative proofs to the Frobenius theorem. The proofs are all in an analytical flavor. In this short note, we give a geomtric proof to the Frobenius theorem by introducing a new canonical connection on a Riemannian manifold with a distribution.
\begin{thm}\label{thm-connection}
Let $(M^n,g)$ be a Riemannian manifold and $E^r$ be a distribution on $M$. Then, there is unique connection $\nabla$ on $M$ compatible with $g$ and with the torsion tensor $\tau$ satisfying the following properties:
\begin{enumerate}
\item $\tau(X,Y)=\tau(\xi,\eta)=0$,
\item $g(\tau(\xi,X),Y)=\frac12(L_\xi g)(X,Y)$,
\item $g(\tau(\xi,X),\eta)=-\frac12(L_X g)(\xi,\eta)$
\end{enumerate}
for any $X,Y\in \Gamma(E)$ and $\xi,\eta\in \Gamma(E^\perp)$.
\end{thm}
When $E=TM$, the connection introduced in Theorem \ref{thm-connection} is just the Levi-Civita connection of $(M,g)$. So, we call the connection $\nabla$ the Levi-Civita connection of $(M,g)$ w.r.t. $E$. A distinguished feature of this connection is that if $E$ is involutive, then $E$ is totally geodesic w.r.t. this connection.
\begin{thm}\label{thm-total-geo}
Let $(M^n,g)$ be a Riemannian manifold and $E^r$ be an involutive distribution on $M$. Let $\nabla$ be the Levi-Civita connection of $(M,g)$ w.r.t. $E$. Then, for any $X,Y\in \Gamma(E)$ and $\xi\in \Gamma(E^\perp)$,
$g(\nabla_XY,\xi)=0$ and equivalently $g(Y,\nabla_X\xi)=0$. In other words,  $E$ is totally geodesic w.r.t. the affine connection $\nabla$.
\end{thm}
From this special feature of the Levi-Civita connection $\nabla$ of $(M,g)$ w.r.t. $E^r$, we can construct the integral submanifolds of $E$ geometrically by using the the exponential map of the connection $\nabla$ when $E$ is involutive. In fact, let $p\in E$ and $X_1,X_2,\cdots X_r$ be a basis of $E_p$, then 
$$\Sigma:=\left\{\exp_p\left(t_1X_1+t_2X_2+\cdots+t_rX_r\right)\ \big|\ |t_1|,|t_2|,\cdots, |t_r|<\epsilon\right\}$$
will give us the integral submanifold passing through $p$ when $\epsilon>0$ is small enough. Moreover, when the connection $\nabla$ is geodesic complete, the geometric construction will also give us some global descriptions of integral submanifolds.

On a Riemannian manifold $(M,g)$ with a distribution $E$, there are two canonical connections named after Schouten-Van Kampen and Vranceanu respectively (see \cite{BF}). Let $\nabla$ be the Levi-Civita connection of $(M,g)$. Then, the Schouten-Van Kampen connection $\nabla^\circ$ and the Vranceanu connection $\nabla^*$ are defined by
\begin{equation*}
\nabla^\circ_XY=(\nabla_{X}Y^\top)^\top+(\nabla_{X}Y^\perp)^\perp  \end{equation*}
and
\begin{equation*}
\nabla^*_XY=(\nabla_{X^\top}Y^\top)^\top+(\nabla_{X^\perp}Y^\perp)^\perp+[X^\perp,Y^\top]^\top+[X^\top,Y^\perp]^\perp \end{equation*}
respectively. Here $X=X^\top+X^\perp$ is the orthogonal decomposition of $X\in \Gamma(TM)$ according to $TM=E\oplus E^\perp$. Both of the Schouten-Van Kampen connection and the Vranceanu connection do not have the property that $E$ is totally geodesic  when it is involutive. 

On a foliation $E$, there is a well-known connection named after Bott defined on the normal bundles along each leaves (see \cite{Bo,GW}). After identifying the normal bundle with $E^\perp$, the Bott connection $D$ is defined as
$$D_X\xi=[X,\xi]^\perp,\ \forall\ X\in \Gamma(E)\mbox{ and }\xi\in \Gamma(E^\perp).$$
This connection is not compatible with the Riemannian metric in general. In fact, 
$$X(g(\xi,\eta))-g(D_X\xi,\eta)-g(\xi,D_X\eta)=(L_Xg)(\xi,\eta)$$
which will not vanish unless $L_Xg|_{E^\perp}=0$. So, the normal part of the Levi-Civita connection $\nabla$ of $(M,g)$ w.r.t. $E$ is different with the Bott connection.

\section{Proofs of the main results}
Let $(M^n,g)$ be a Riemannian manifold. For an affine connection $D$, the torsion tensor of $D$ is defined as
$$\tau^D(X,Y)=D_XY-D_YX-[X,Y]$$
which is a vector-valued two-form or more precisely a smooth section of $\Hom(\wedge^2TM,TM)$. The fundamental theorem of Riemannian geometry says that their is a unique torsion free affine connection compatible with $g$ on the Riemannian manifold $(M,g)$ (see \cite{CE} for example). The theorem has an extension to compatible connections with general torsion tensors whose proof is the same as the proof of the fundamental theorem of Riemannian geometry in \cite{CE}. So, we only collect the statement below and omit its proof.
\begin{thm}\label{thm-fund-Riem}
Let $(M^n,g)$ be a Riemannian manifold and $\mathscr A(M,g)$ be the space of affine connections on $M$ compatible with the metric $g$. Then, the map
$$\tau:\mathscr A(M,g)\to \Gamma(\Hom(\wedge^2 TM,TM)):D\mapsto \tau^D$$
is an affine bijection.  More precisely, for any $\tau\in \Gamma(\Hom(\wedge^2 TM,TM))$, the inverse of $\tau$ is $D^\tau\in \mathscr A(M,g)$ defined by
\begin{equation*}
\begin{split}
&2g(D^\tau _XY,Z)\\
=&X(g(Y,Z))-Z(g(X,Y))+Y(g(Z,X))+g([Z,X],Y)-g([Y,Z],X)+g([X,Y],Z)\\
&+g(\tau(Z,X),Y)-g(\tau(Y,Z),X)+g(\tau(X,Y),Z)
\end{split}
\end{equation*}
for all $X,Y,Z\in \Gamma(TM)$.
\end{thm}
\begin{rem}
In Theorem \ref{thm-fund-Riem}, note that $\mathscr A(M,g)$ is natrually an affine space. For that the map $\tau$ is affine, we mean that 
$$\tau^{\lambda_0D_0+\lambda_1D_1}=\lambda_0\tau^{D_0}+\lambda_1\tau^{D_1}$$
for any $D_0,D_1\in \mathscr A(M,g)$ and $\lambda_0,\lambda_1\in \R$ with $\lambda_0+\lambda_1=1$.
\end{rem}

 We need the following simple lemma to  prove Theorem \ref{thm-connection}.
\begin{lem}\label{lem-tensor}
Let $(M^n,g)$ be a Riemannian manifold and $E^r$ be a distribution on $M$. Let
$$T(\xi,X,Y):=(L_{\xi}g)(X,Y), \forall \xi\in \Gamma(E^\perp)\mbox{ and } \forall X,Y\in \Gamma(E).$$
Then, $T\in \Gamma(\Hom(E^\perp\otimes E\otimes E,\R)).$
\end{lem}
\begin{proof}
We only need to verify that $T$ is $C^\infty(M)$-linear w.r.t. its three arguments. The conclusion then follows by direct computation. 
\end{proof}
We are now ready to prove Theorem \ref{thm-connection} by using Theorem \ref{thm-fund-Riem}.
\begin{proof}[Proof of Theorem \ref{thm-connection}]
 Note that 
 \begin{equation*}
 \begin{split}
 &\Hom(\wedge^2TM,TM)\\
 =&\Hom(\wedge^2E,TM)\oplus\Hom(\wedge^2E^\perp,TM)\oplus \Hom(E^\perp\otimes E,E)\oplus\Hom(E^\perp\otimes E,E^\perp).
 \end{split}
  \end{equation*}
 So, to give a vector-valued two-form, we only need to specify its four components according to the above decomposition. Let $\tau$ be the vector-valued two-form with the four components specified by (1),(2) and (3) in the statements of Theorem \ref{thm-connection}. By Lemma \ref{lem-tensor}, we know that (2) and (3) specify the $\Hom(E^\perp\otimes E,E)$-part and $\Hom(E^\perp\otimes E, E^\perp)$-part of $\tau$ respectively. Then, by Theorem \ref{thm-fund-Riem}, we complete the proof of the theorem.
\end{proof}

We next come to prove Theorem \ref{thm-total-geo}.
\begin{proof}[Proof of Theorem \ref{thm-total-geo}]
By Theorem \ref{thm-fund-Riem}, we have
\begin{equation*}
\begin{split}
&2g(\nabla_XY,\xi)\\
=&X(g(Y,\xi))-\xi(g(X,Y))+Y(g(\xi,X))+g([\xi,X],Y)-g([Y,\xi],X)+g([X,Y],\xi)\\
&+g(\tau(\xi,X),Y)-g(\tau(Y,\xi),X)+g(\tau(X,Y),\xi)\\
=&-(L_\xi g)(X,Y)+g(\tau(\xi,X),Y)+g(\tau(\xi,Y),X)\\
=&0.
\end{split}
\end{equation*}
This completes the proof of the theorem.
\end{proof}
\begin{rem}
Note that requirement (3) in Theorem \ref{thm-connection} is not used in the proof of Theorem \ref{thm-total-geo}. We take this requirement only for the purpose making $E$ and $E^\perp$ have equal positions.
\end{rem}
In the geometric proof of the Frobenius theorem, we need the following two lemmas giving some properties of the Levi-Civita connection on $(M,g)$ w.r.t. to an involutive distribution $E$.
\begin{lem}\label{lem-curve}
Let $(M^n,g)$ be a Riemannian manifold with an involutive distribution $E^r$ and $\nabla$ be the Levi-Civita connection of $(M,g)$ w.r.t. $E$. Then,
\begin{enumerate}
\item for any smooth curve $\gamma:[0,a]\to M$ with $\gamma(0)\in E$ and $\nabla_{\gamma'}\gamma'(t)\in E$ for $t\in [0,a]$, we have $\gamma'(t)\in E$ for any $t\in [0,a]$;
\item let $\gamma:[0,a]\to M$ be a smooth curve with $\gamma'(t)\in E$ for $t\in [0,a]$ and $V$ be a smooth vector field along $\gamma$ with $V(0)\in E$ and $\nabla_{\gamma'}V(t)\in E$ for $t\in [0,a]$. Then, $V(t)\in E$ for $t\in [0,a]$;
\item let $\gamma:[0,a]\to M$ be a smooth curve with $\gamma'(t)\in E$ for $t\in [0,a]$ and $V$ be a smooth vector field along $\gamma$ with $V(0)\in E^\perp$ and $\nabla_{\gamma'}V(t)\in E^\perp$ for $t\in [0,a]$. Then, $V(t)\in E^\perp$ for $t\in [0,a]$.
\end{enumerate}
\end{lem}
\begin{proof}
\item[(1)] Let $X_1,X_2,\cdots, X_r$ and $X_{r+1}, X_{r+2},\cdots, X_n$ be local frames of $E$ and $E^\perp$ at $\gamma(0)$. Suppose that
$$\gamma'(t)=\sum_{i=1}^na_i(t)X_i\mbox{ and }\nabla_{X_j}X_i=\sum_{k=1}^n\Gamma_{ij}^kX_k.$$
Then, by Theorem \ref{thm-total-geo}, we have $\Gamma_{ij}^\alpha=0$ when $1\leq i,j\leq r$ and $r+1\leq\alpha\leq n$. Moreover, by that 
$$\nabla_{\gamma'}\gamma'=\sum_{i=1}^n\left(a'_i+\sum_{k,l=1}^na_ka_l\Gamma_{kl}^i\right)X_i\in E.$$
We have
\begin{equation}\label{eq-a}
a'_\alpha+\sum_{\beta={r+1}}^nA_{\alpha\beta} a_\beta=0\ \mbox{for $\alpha=r+1,r+2,\cdots,n$}
\end{equation}
where 
$$A_{\alpha\beta}=\sum_{k=1}^ra_k\left(\Gamma_{k\beta}^\alpha+\Gamma_{\beta k}^\alpha\right)+\sum_{\lambda=r+1}^{n}a_\lambda\Gamma_{\lambda\beta}^\alpha.$$
Note that \eqref{eq-a} is a homogeneous linear system of ODEs on $a_\alpha$ with $r+1\leq \alpha\leq n$ and $a_\alpha(0)=0$ for $r+1\leq\alpha\leq n$ by that $\gamma'(0)\in E$. Thus $a_\alpha\equiv 0$ and hence $\gamma'\in E$. This completes the proof of (1).
\item[(2)] Let $X_{r+1},X_{r+2},\cdots,X_{n}$ be a smooth frame of $E^\perp$ along $\gamma$. By Theorem \ref{thm-total-geo}, $\nabla_{\gamma'}X_\alpha\in E^\perp$ for $r+1\leq\alpha\leq n$. Suppose that
$$\nabla_{\gamma'}X_{\alpha}=\sum_{\beta=r+1}^nB_{\alpha\beta}X_\beta.$$
Then
    $$\vv<V,X_\alpha>'=\vv<\nabla_{\gamma'}V,X_\alpha>+\vv<V,\nabla_{\gamma'}X_\alpha>=\sum_{\beta=r+1}^nB_{\alpha\beta}\vv<V,X_\beta>$$
for $\alpha=r+1,r+2,\cdots,n$. This forms a homogeneous linear system of ODEs on $\vv<V,X_\alpha>$ with $\alpha=r+1,r+2,\cdots,n$. Noting that $\vv<V,X_\alpha>(0)=0$, we have $\vv<V,X_\alpha>\equiv 0$ and hence $V(t)\in E$ for $t\in [0,a]$. This completes the proof of (2).
\item[(3)] The proof is the same as that of (2).
\end{proof}
\begin{lem}\label{lem-exp}
Let $(M^n,g)$ be a Riemannian manifold with an involutive distribution $E^r$ and $\nabla$ be the Levi-Civita connection on $(M,g)$ w.r.t. $E$. Then, for any $p\in M$ and $X\in E_p$ where $\exp_p$ is defined, 
$\left(\exp_p\right)_{*X}(Y)\in E$ for any $Y\in E_p$.  
\end{lem}
\begin{proof}
Let $\gamma(t)=\exp_p(tX)$ with $t\in [0,1]$. By (1) of Lemma \ref{lem-curve}, $\gamma'(t)\in E$ for any $t\in [0,1]$. Let $v_1,v_2,\cdots,v_n$ be an orthonormal basis of $T_pM$ such that $v_i\in E_p$ for $1\leq i\leq r$ and $v_\alpha\in E_p^\perp$ for $r+1\leq \alpha\leq n$. Let $V_i$ be the parallel translation of $v_i$ along $\gamma$ for $1\leq i\leq n$. Then, by (2) and (3) of Lemma \ref{lem-curve}, we know that $V_i(t)\in E$ and $V_\alpha(t)\in E^\perp$ for $t\in [0,1]$, $1\leq i\leq r$ and $r+1\leq \alpha\leq n$ respectively.

Let $\Phi(u,t)=\exp_p(t(X+uY))$ with $t\in [0,1]$ and $u\in (-\epsilon,\epsilon)$ for some $\epsilon>0$ small enough. It is a geodesic variation of $\gamma$. So its variation field $$J(t)=(\exp_p)_{*tX}(tY)\mbox{ for }t\in [0,1]$$
is a Jacobi field along $\gamma$ with $J(0)=0$ and $J'(0)=Y\in E_p$. It satisfies the Jacobi field equation:
$$\nabla_{\gamma'}\nabla_{\gamma'} J=R(\gamma',J)\gamma'+\nabla_{\gamma'}\left(\tau(\gamma',J)\right).$$
Let $J=\sum_{i=1}^nJ_iV_i$. Then, for $\alpha=r+1,r+2,\cdots,n$, we have
\begin{equation}\label{eq-J}
\begin{split}
J''_\alpha=&\vv<J,V_\alpha>''\\
=&\vv<J'',V_\alpha>\\
=&\vv<R(\gamma',J)\gamma',V_\alpha>+\vv<\nabla_{\gamma'}\left(\tau(\gamma',J)\right),V_\alpha>\\
=&\sum_{\beta=r+1}^n\vv<R(\gamma',V_\beta)\gamma',V_\alpha>J_\beta+\vv<\tau(\gamma',J),V_\alpha>'\\
=&\sum_{\beta=r+1}^n\vv<R(\gamma',V_\beta)\gamma',V_\alpha>J_\beta+\left(\sum_{\beta=r+1}^n\vv<\tau(\gamma',V_\beta),V_\alpha>J_\beta\right)'\\
=&\sum_{\beta=r+1}^n\vv<\tau(\gamma',V_\beta),V_\alpha>J'_\beta+\sum_{\beta=r+1}^n\left(\vv<R(\gamma',V_\beta)\gamma',V_\alpha>+\vv<\tau(\gamma',V_\beta),V_\alpha>\right)J_\beta.
\end{split}
\end{equation}
Here, we have used the fact that $R(X,Y)Z\in E$ for any $X,Y,Z\in E$ since $E$ is totally geodesic by Theorem \ref{thm-total-geo} and (1) of Theorem \ref{thm-connection}. 

Therefore, \eqref{eq-J} form a linear homogeneous system of ODEs for $J_\alpha$ with $r+1\leq J_\alpha\leq n$. Noting that $J_\alpha(0)=J'_\alpha(0)=0$ for $r+1\leq\alpha\leq n$ since $J(0)=0$ and $J'(0)=Y\in E$. Thus $J_\alpha\equiv 0$ for $r+1\leq\alpha\leq n$ and hence $J(t)\in E$ for any $t\in [0,1]$. In particular, $$J(1)=\left(\exp_p\right)_{*X}(Y)\in E.$$ This completes the proof of the lemma.
\end{proof}
Finally, we come to prove the Frobenius theorem.
\begin{thm}
Let $E^r$ be an involutive distribution on the smooth manifold $M^n$. Then, for any $p\in M$, there is a local coordinate $(U,x)$ at $p$ such that $$E|_U=\Span\left\{\frac{\p}{\p x^1},\frac{\p}{\p x^2},\cdots,\frac{\p}{\p x^r}\right\}.$$
\end{thm}
\begin{proof}
Let $g$ be Riemannian metric on $M$ and $\nabla$ be the Levi-Civita connection of $g$ w.r.t. $E$. Let $X_1,X_2,\cdots, X_r$ and $X_{r+1},\cdots,X_n$ be local frames of $E$ and  $E^\perp$ on some open neighborhood of $p$ respectively. Consider the map $\Phi:B_o(\delta)\to M$ defined as
$$\Phi(x)=\exp_{\exp_p\left(x^{r+1}X_{r+1}+\cdots+x^nX_n\right)}\left(x^1X_1+\cdots x^rX_r\right).$$
Here $B_o(\delta)$ is a ball of $\R^n$ centered at the origin $o$ and $\delta>0$ is small enough so that the map $\Phi$ is defined. It is clear that $$\Phi_{*o}\left(\frac{\p}{\p x^{i}}\right)=X_i(p)$$ for $i=1,2,\cdots, n$. By the inverse function theorem, when $\delta$ is small enough, $\Phi$ is a smooth embedding. Let $U=\Phi(B_o(\delta))$ and $x=\Phi^{-1}$. Then, by Lemma \ref{lem-exp}, this is the required local coordinate.
\end{proof}

\end{document}